\documentclass [10pt,reqno]{amsart}
\usepackage{ucs}
\usepackage{amsmath, amssymb, amscd}
\usepackage{pstricks}
\usepackage{sseq}
\usepackage[latin1]{inputenc}
\usepackage{amsmath}
\usepackage{amssymb}
\usepackage{epsfig}
\usepackage{graphicx}
\usepackage{psfrag}

\usepackage{tikz}
\usepackage{tikz-cd}
\usetikzlibrary{matrix,arrows,decorations.pathmorphing}

\newtheorem{theorem}{Theorem}[section]
\newtheorem{definition}[theorem]{Definition}
\newtheorem{lemma}[theorem]{Lemma}

\newtheorem{prop}[theorem]{Proposition}
\newtheorem{proposition}[theorem]{Proposition}

\newtheorem{conjecture}[theorem]{Conjecture}

\newtheorem*{remark}{Remark}

\newcommand{\abs}[1]{\lvert#1\rvert}

\renewcommand{\epsilon}{\varepsilon}

\DeclareMathAlphabet{\mathpzc}{OT1}{pzc}{m}{it}
\usepackage{mathrsfs}

\newcommand{\N}{\mathbb{N}}
\newcommand{\Z}{\mathbb{Z}}

\newcommand{\Q}{\mathbb{Q}}

\renewcommand{\qed}{$\hfill \square$ \smallskip \\}

\newcommand{\Spinc}{Spin^c}

\newcommand{\Hom}{\text{Hom}}

\renewcommand{\det}{\text{det}}

\newcommand{\tr}{\text{tr}}

\newcommand{\rk}{\operatorname{rk}}

\newcommand{\HF}{\widehat{\text{\em HF}}}
\newcommand{\Khr}{\text{\em Khr}}

\newcommand{\KHI}{\operatorname{KHI}}

\begin{document}
\thispagestyle{empty}
\title[On spectral sequences from Khovanov homology]{On spectral sequences from Khovanov homology}
\author{Andrew Lobb \\ Raphael Zentner}

\begin {abstract} There are a number of homological knot invariants, each satisfying an unoriented skein exact sequence, which can be realized as the limit page of a spectral sequence starting at a version of the Khovanov chain complex.  Compositions of elementary $1$-handle movie moves induce a morphism of spectral sequences.  These morphisms remain unexploited in the literature, perhaps because there is still an open question concerning the naturality of maps induced by general movies.

In this paper we focus on the spectral sequences due to Kronheimer-Mrowka from Khovanov homology to instanton knot Floer homology, and on that due to Ozsv\'ath-Szab\'o to the Heegaard-Floer homology of the branched double cover.  For example, we use the $1$-handle morphisms to give new information about the filtrations on the instanton knot Floer homology of the $(4,5)$-torus knot, determining these up to an ambiguity in a pair of degrees; to determine the Ozsv\'ath-Szab\'o spectral sequence for an infinite class of prime knots; and to show that higher differentials of both the Kronheimer-Mrowka and the Ozsv\'ath-Szab\'o spectral sequences necessarily lower the delta grading for all pretzel knots.
\end {abstract}

\address{Department of Mathematical Sciences \\ Durham University \\ Durham DH1 3LE \\ UK}
\email{andrew.lobb@durham.ac.uk}
\address {Mathematisches Institut \\ Universit\"at Regensburg \\  D-90340 Regensburg \\  Deutschland}
\email{raphael.zentner@mathematik.uni-regensburg.de}

\maketitle

\section{Introduction}
Recent work in the area of the $3$-manifold invariants called \emph{knot homologies} has illuminated the relationship between Floer-theoretic knot homologies and `quantum' knot homologies.  The relationships observed take the form of spectral sequences starting with a quantum invariant and abutting to a Floer invariant.  A primary example is due to Ozsv\'ath and Szab\'o \cite{OSb2c} in which a spectral sequence is constructed from Khovanov homology of a knot (with $\Z/2$ coefficients) to the Heegaard-Floer homology of the 3-manifold obtained as double branched cover over the knot.  A later example is due to Kronheimer and Mrowka which gives a spectral sequence \cite{KM_ss, KM_filtrations} from Khovanov homology to an instanton knot Floer homology.

There are automatically naturality questions about such spectral sequences.  Both the quantum homology and the Floer homology involved exhibit some functoriality with respect to link cobordism, and one can ask if the spectral sequences behave well with respect to this functoriality.  The project of demonstrating such naturality is important, but in this paper we are able to use the limited naturality already available (essentially naturality for cobordisms presented as a movie of elementary $1$-handle additions) to make some computations.  The basic idea is that if we are interested in the Floer homology of a knot $K$, we find a cobordism to a knot $K'$ with a simple spectral sequence and then use the quantum homology of $K'$ to draw conclusions on the Floer homology of $K$.

We are restricting ourselves to the spectral sequences of Ozsv\'ath-Szab\'o and Kronheimer-Mrowka, but the technique should have wider applicability.  In the next section we review these Floer homologies, Section 3 then deals with the spectral sequences, and Section 4 contains the computations.

A word of warning: As a matter of notational convenience, our Floer theoretic invariant of a knot or link or 3-manifold is really what in the literature would be the Floer invariant of the mirror image of a knot or link or 3-manifold.  This avoids permanent use of the word `mirror' in the spectral sequences that we study.

\subsection*{Acknowledgements}
We thank Peter Kronheimer for helpful email correspondence. We also thank CIRM where we started this project in a research in pairs program in March this year.  The first author thanks Liam Watson.
 The first author was partially supported by EPSRC grant EP/K00591X/1.

\section{Review of Heegaard-Floer and instanton Floer homology}

While Khovanov homology is very simply defined and Heegaard-Floer homology for many is a relatively comfortable object, instanton Floer homology is far less known.  Therefore we are going to assume familiarity with Khovanov homology and devote the first subsection merely to quoting a result from Heegaard-Floer, while the remaining subsections give a review of the relevant instanton Floer homology.  We will work with the reduced homology theories.

\subsection{Heegaard-Floer homology}

In this paper we are concerned with $\HF$, the `hat' version of Heegaard-Floer homology \cite{OS_first}.  This is an invariant of a closed $3$-manifold equipped with a $\Spinc$-structure and takes the form of a finitely-generated vector space over $\Z/2$.  We are interested in $3$-manifolds $\Sigma(L)$ that are obtained as branched double-covers over the mirror images of links $L \subset S^3$, and, taking the sum over all $\Spinc$-structures, we regard $\HF$ simply as a vector space.

\begin{theorem}[Ozsv\'ath-Szab\'o \cite{OSb2c}]
\label{OS_ss_thm}
Given a link $L \subset S^3$, there is a spectral sequence (which \emph{a priori} depends on a choice of link diagram) abutting to $\HF(\Sigma(L))$ with $E_1$-page equal to the reduced Khovanov chain complex and $E_2$-page equal to the reduced Khovanov homology $\Khr(L)$ (where everything has been taken with $\Z/2$ coefficients).
\end{theorem}

In general this theorem implies that the rank of $\HF(\Sigma(K))$ is bounded above by the rank of $\Khr(K)$.

For a knot $K$ the number of $\Spinc$ structures on $\Sigma(K)$ is equal to $|{\rm det}(K)|$, from which by an Euler characteristic argument it follows that the rank of $\HF(\Sigma(K))$ is bounded below by $|{\rm det}(K)|$, and when this bound is tight $\Sigma(K)$ is called an $L$-space.   It is a quick check that if $K$ is a knot with thin Khovanov homology then the rank of $\Khr(K)$ is exactly $|\det(K)|$ and hence the spectral sequence collapses at the $E_2$-page.

Computations of non-trivial spectral sequences for specific prime knots were given by Baldwin \cite{Baldwin}, and he observed that the spectral sequences he found had differentials that strictly decreased the $\delta$-grading on Khovanov homology.  Later in this paper we extend Baldwin's examples to an infinite class of prime knots and furthermore show that the Ozsv\'ath-Szab\'o spectral sequence has differentials which strictly decrease the $\delta$-grading for all pretzel knots.

\subsection{Instanton knot Floer homology}

Instanton knot Floer homology as constructed by Kronheimer and Mrowka \cite{KM_sutures, KM_ss, KM_filtrations} is an invariant of pairs consisting of links in 3-manifolds.  In the manifestation that interests us, we shall be restricting our attention to the case of knots and links inside the 3-sphere $K \subset S^3$.

Reduced instanton knot Floer homology $I^\natural(K)$ of a link $K$ with a marked component in the 3-sphere $S^3$ is, roughly speaking, defined via the Morse homology of a Chern-Simons functional on a space of connections that have a prescribed asymptotic holonomy around the link $K$ \cite{KM_filtrations}. It is an abelian group with an absolute $\Z/4$ grading \cite{KM_ss} (usually instanton Floer homology comes with relative $\Z/4$ gradings, but in \cite[Section 4.5 and Section 7.4]{KM_ss} absolute gradings are given). We denote by $(C(K)^\natural,d^\natural)$ the $\Z/4$ graded complex whose homology is $I^\natural(K)$. The differential $d^\natural$ lowers the $\Z/4$ grading by $1$.  Involved in the construction of this complex are various choices of perturbations one has made, but we have suppressed these in the notation as our computations will not use the definition. 
\\
\begin{remark}
	For a matter of notation we denote by $I^\natural(K)$ what Kronheimer-Mrowka denote as $I^\natural(\overline{K})$, the reduced instanton Floer homology of the mirror image of $K$. 
\end{remark}

Kronheimer and Mrowka have shown in \cite{KM_ss,KM_filtrations} that this can also be computed from the Khovanov cube as we shall now recall.

We are assuming familiarity with reduced Khovanov homology.  
Given a marked link $K$ with a diagram $D$ we shall denote the reduced Khovanov chain complex by $(C(D),d_{\Khr}(D))$ whose homology is the reduced Khovanov homology $\Khr(K)$ of $K$. The vector space $C(D)$ is a bigraded complex $(C(D)^{i,j})$, where $i$ denotes the {\em homological} and $j$ denotes the {\em quantum} grading. The differential $d_{\Khr}$ is bigraded of degree $(1,0)$. The two gradings also define a descending filtration $\mathscr{F}^{i,j} C(D)$ indexed by $\Z \oplus \Z$. With respect to this filtration a morphism $\phi$ is said to be of order $ \geq (s,t)$ if $\phi(\mathscr{F}^{i,j}C(D)) \subseteq \mathscr{F}^{i+s,j+t}C(D)$.

We follow the standard convention that gives the reduced Khovanov chain complex as a subcomplex of the Khovanov chain complex.  This has the unfortunate effect that the reduced Khovanov homology of the unknot is one copy of the ground ring (for us either $\Q$ or $\Z/2$) supported in gradings $i = 0$ and $j= -1$ (where one might think $j = 0$ more natural).  Nevertheless this brings us in line with most current usage.

The first statement of the following Theorem appears as \cite[Theorem 6.8]{KM_ss}, the second appears as \cite[Theorem 1.1]{KM_filtrations} for the unreduced versions. A corresponding statement is true for the reduced homologies. 

\begin{theorem}\label{instanton cube}(Kronheimer-Mrowka) Let $D$ be a diagram of a knot or link $K$. 
     \begin{enumerate}
     	\item[(i)] There is a differential $d_\natural(D)$ on the vector space $C(D)$ whose homology is isomorphic to the reduced instanton knot Floer homology $I^\natural(K)$. More precisely, the bigrading $(i,j)$ gives a $\Z/4$ grading on $C(D)$ by $ j-i - 1 \mod{4}$. The differential $d_\natural(D)$ lowers this $\Z/4$ grading by $1$. Then there is a quasi-isomorphism of $\Z/4$ graded chain complexes $(C(K)^\natural,d^\natural) \to (C(D),d_\natural(D))$. \\
		\item[(ii)] The difference $d_\natural(D) - d_{\Khr}(D)$ is filtered of degree $(1,2)$. 
     \end{enumerate}
\end{theorem}
As a consequence of the second point, both the homological and the quantum filtrations on the Khovanov complex induce a filtration on the instanton knot Floer homology. This is a novum of \cite{KM_filtrations} compared to \cite{KM_ss}. 
A priori these filtrations might depend on the chosen diagram, but it is not the case. In fact, Kronheimer and Mrowka have shown that the induced filtrations are invariants of the link $K$ \cite[Theorem 1.2 and Corollary 1.3]{KM_filtrations}, therefore yielding the following result:

\begin{theorem}\cite[Theorem 1.2 and Corollary 1.3]{KM_filtrations}
\label{ss}
Let $K$ be a link and let $D$ be a diagram of $K$. 
Let $a,b \geq 1$. The descending filtrations induced by $a i + b j$ on $C(D)$ is preserved by $d_\natural(D)$, and the induced filtration on $I^\natural(K)$ depends on the link $K$ only. The pages of the associated Leray spectral sequence $(E_r,d_r)$, converging to $I^\natural(K)$, are invariants of $K$ for $r \geq a+1$. There are no differentials before the $E_a$ page, and the page $E_{a+1}$ is the reduced Khovanov homology of $K$. 
\end{theorem}

For instance, the homological filtration induces a spectral sequence abutting to $I^\natural(K)$ whose $E_2$ page is Khovanov homology, and the quantum filtration induces a spectral sequence whose $E_1$ page is Khovanov homology.

The statement in the last sentence is not explicit in \cite{KM_filtrations} but is easily checked from Kronheimer and Mrowka's Theorem \ref{instanton cube}.

\subsection{The Alexander polynomial}
In \cite{KM_sutures} Kronheimer and Mrowka developed an instanton Floer homology of sutured manifolds, yielding a $\Z/4$ graded link homology group $\KHI(K)$ of a link $K$. In \cite{KM_alex,Lim} it is shown by Kronheimer and Mrowka, and independently by Lim, that this is related to the Alexander polynomial. In fact, $\KHI(K)$ carries two commuting operators whose common eigenspace decompositions give $\KHI(K)$ a $\Z \oplus \Z/2$ grading. For a knot, the ``graded Euler characteristic'' of $\KHI(K)$ is equal to minus the Conway-normalised Alexander polynomial $\Delta_K$:
\begin{equation}\label{alex polynomial}
	- \Delta_K(t) = \sum_{h \in \Z,i \in \Z/2} (-1)^{i} t^{h} \rk(\KHI^{i,h}(K))
\end{equation}

\begin{proposition}\label{alex polynomial lower bound} \cite{KM_ss}
	For {\em knots} $K$ Floer's excision theorem yields a natural isomorphism $I^\natural(K) \cong \KHI(K)$. As a consequence, $I^\natural(K)$ is in rank bounded below by the sum of the absolute values of the coefficients of the Alexander polynomial $\Delta_K$.
\end{proposition}

\subsection{Thin Khovanov homology}
The reduced Khovanov homology of an oriented link $L$ is a bigraded vector space over the rational numbers $\Khr(L)$ which categorifies the Jones polynomial $V_L(q)$, normalised such that for the unknot $U$ one has $V_U(q)=q^{-1}$. More precisely, one has the formula
\begin{equation}\label{Khovanov and Jones}
	\sum_{i,j \in \Z} (-1)^{i} q^j \, \rk(\Khr^{i,j}(L))  = V_L(q) {\rm ,} 
\end{equation}
see for instance \cite{Khovanov_patterns}.

A link $L$ is said to have {\em thin} Khovanov homology if all non-trivial vector spaces $\Khr^{i,j}(L)$ occur on one line where $j-2i$ is constant. Kronheimer and Mrowka have shown that their spectral sequence from reduced Khovanov homology $\Khr(K)$ to reduced instanton knot Floer homology $I^\natural(K)$ has no non-trivial differential after the Khovanov page if $K$ is a quasi-alternating knot, see \cite[Corollary 1.6]{KM_ss}. Their result can easily be strengthened a little bit.

\begin{proposition} \label{knots with thin homology}
	Suppose that $K$ is a knot that has thin reduced Khovanov homology. Then the Kronheimer-Mrowka spectral sequence $\Khr(K) \Rightarrow I^\natural(K)$ has no non-zero differential (over $\Q$). The total rank of $\Khr(K)$ and $I^\natural(K)$ then agree with the determinant of $K$ given by $\abs{\Delta_K(-1)} = \abs{V_K(-1)}$. 
\end{proposition}
\begin{proof}
	Suppose the reduced Khovanov homology $\Khr^{i,j}(K)$ of $K$ is supported on the line $j = 2i + s$ for some even integer $s$. Then from formula (\ref{Khovanov and Jones}) above it follows that 
\begin{equation*} \begin{split}
	V_K(-1) & = V_K(\sqrt{-1}^2) = \sum_{i \in \Z} (-1)^i (-1)^{i + s/2} \rk(\Khr^{i,2i + s}(K)) \\ & = (-1)^{s/2} \rk(\Khr(K)) \ .
\end{split}
\end{equation*}
Therefore, the determinant is equal to the total rank of the reduced Khovanov homology of $K$. On the other hand, Proposition \ref{alex polynomial lower bound} above gives the same lower bound. As therefore the rank of reduced Khovanov homology $\Khr(K)$ and $I^\natural(K)$ have to coincide, there is no non-zero differential in the spectral sequence.
\end{proof}

An analoguous result holds in Heegaard-Floer homology: knots with thin Khovanov homology have branched double covers which are Heegaard-Floer $L$-spaces.

As a consequence of the spectral sequence, the total rank of Khovanov homology provides an upper bound for the rank of instanton homology.  In the case where these ranks agree, all the information about the filtration on instanton homology is contained in Khovanov homology.

\begin{figure}
\centerline{
{
\psfrag{K0}{$K_0$}
\psfrag{K1}{$K_1$}
\psfrag{K2}{$K_2$}
\psfrag{thing4}{$-q^n$}
\psfrag{2}{$2$}
\psfrag{x4}{$x_4$}
\psfrag{T+(D)}{$T^+(D)$}
\includegraphics[height=2in,width=3in]{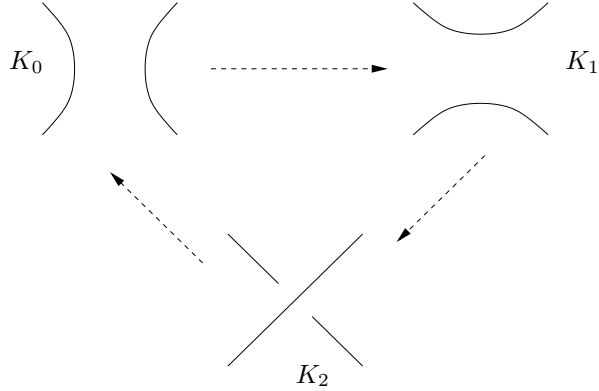}
}}
\caption{The links $K_0$, $K_1$, $K_2$ comprise an unoriented skein triple.}
\label{smooth}
\end{figure}

\subsection{Unoriented skein exact triangles}
Both Khovanov homology and instanton knot Floer homology have unoriented skein exact triangles of which we shall make extensive use in our computational section.
\begin{proposition} \label{triangle} Suppose $K_0$, $K_1$, and $K_2$ are three links with diagrams $D_0$, $D_1$, and $D_2$ respectively that look the same except near a crossing of $D_2$ where they differ as in Figure \ref{smooth}. Then there is a long exact triangle relating the groups $I^\natural(K_0)$, $I^\natural(K_1)$, and $I^\natural(K_2)$, and likewise for the reduced Khovanov homology groups $\Khr(K_0)$, $\Khr(K_1)$, and $\Khr(K_2)$:
\begin{equation*}
\begin{tikzcd} \tiny
I^\natural(K_0) \arrow{rr} & & I^\natural(K_1) \arrow{ld}  & \Khr(K_0) \arrow{rr} & & \Khr(K_1) \arrow{ld} \\
	& I^\natural(K_2) , \arrow{ul} &  & & \Khr(K_2) . \arrow{ul} &
\end{tikzcd}
\end{equation*}
All maps are induced by standard cobordisms corresponding to 1-handle attachment in both theories.
\end{proposition}

\subsection{Knot cobordisms and functoriality}
The instanton knot Floer homology $I^\natural(K)$ groups are functorial for knot and link cobordisms \cite{KM_ss}: Given two oriented links $K_0$ and $K_1$, and a cobordism (not necessarily orientable) $S \subseteq [0,1] \times S^3$ from $K_0 \subseteq \{ 0 \} \times S^3$ to $K_1 \subseteq \{ 1 \} \times S^3$, there is an induced morphism $I^\natural(S) : I^\natural(K_0) \to I^\natural(K_1)$ which is well defined up to an overall sign. Furthermore, the morphism induced by a composite cobordism is the composite of the morphisms of the cobordisms.

\begin{proposition}\cite[Proposition 1.5]{KM_filtrations}\label{filtrations cobordism map}
Let $S$ be a cobordism from a link $K_0$ to a link $K_1$. Let $D_0$ and $D_1$ be diagrams for $K_0$ and $K_1$. Then the map $I^\natural(S) : I^\natural(K_0) \to I^\natural(K_1)$ is induced by a chain map $c: C(D_0) \to C(D_1)$ which has order 
\begin{equation*}
	\geq \left( \frac{1}{2} (S \cdot S) , \chi(S) + \frac{3}{2} (S \cdot S)\right) \ ,
\end{equation*}
where $\chi(S)$ denotes the Euler characterstic of $S$, and $S \cdot S$ denotes the self-intersection number of $S$ with the boundary condition that a push-off at the ends is required to have linking number $0$ with $K_0$ respectively $K_1$.
\end{proposition}

In general, a movie $M$ between diagrams $D_0$ and $D_1$ consisting of 0-, 1- and 2-handle attachments, and of Reidemeister moves, induces a cobordism $S_M$ between the corresponding knots $K_0$ and $K_1$. Such a movie induces a morphism $c(M): C(D_0) \to C(D_1)$ between the corresponding Khovanov complexes by composing the Khovanov morphisms from handle attachments and the chain homotopy equivalences coming from the Reidemeister moves in the respective order. In particular, there is a resulting map $\Khr(M)$ from the Khovanov homology of $K_0$ to $K_1$.
\\

\section{Constraints on Floer homology}
In this section we show how conclusions on the Kronheimer-Mrowka or Ozsv\'ath-Szab\'o spectral sequences for a specific knot or link might be made from link cobordisms.

\subsection{Morphisms of spectral sequences}
Given two spectral sequences $(E_r,d_r)$ and $(E'_r,d'_r)$, a collection of morphisms $(f_r: E_r \to E'_r)$ is said to be a morphism of spectral sequences if 
\begin{itemize} 
\item for any $r$ the morphism $f_r$ is a morphism of chain complexes from the complex $(E_r,d_r)$ to $(E'_r,d'_r)$, i.e. $f_r$ intertwines the differentials $d_r$ and $d'_r$, and
\item the morphism $f_{r+1}$ is the morphism induced by $f_r$ on homology under the isomorphisms $H(E_r,d_r) \cong E_{r+1}$ and $H(E'_r,d'_r) \cong E'_{r+1}$, for any $r \in \N$. 
\end{itemize}
For instance, having filtered complexes $(C,d)$ and $(C',d')$, filtered by families $(\mathscr{F}^nC)_n$ and $(\mathscr{G}^nC')_n$, and a morphism of chain complexes $f:C \to C'$ that respects the filtrations -- meaning that $f(\mathscr{F}^n C) \subseteq \mathscr{G}^{n} C'$ for all $n$ -- the map $f$ induces a morphism between the two spectral sequences coming from the filtrations. 

\begin{definition}
We say that an element $x \in E_s$ is an $s$-boundary if $x$ is in the image of $d_s$, and we say that $x$ is an $s$-cycle if $d_s(x) = 0$.  We say that an element $x \in E_s$ {\em is an $\infty$-cycle} if $x$ lies in the kernel of $d_s$, and its homology class $[x]_{t}$ is a $(t+1)-cycle$ for all $t \geq s$.

\end{definition}

\begin{lemma}\label{algebraic lemma}
 Let $(f_r): (E_r,d_r) \to (E'_r,d'_r)$ be a morphism of spectral sequences. 
 \begin{enumerate}
 	\item[(i)] If $x \in E_s$ is an $s$-cycle then $f_s(x) \in E'_s$ is an $s$-cycle. 
	\item[(ii)] If $x \in E_s$ is an $s$-boundary then $f_s(x) \in E'_s$ is an $s$-boundary.
	\item[(iii)] If $x \in E_s$ is an $\infty$-cycle then $f_s(x) \in E'_s$ is an $\infty$-cycle.
\end{enumerate}
\end{lemma}
\begin{proof}
The result follows from the fact that morphisms of chain complexes preserve cycles and boundaries.
\end{proof}

A chain map $c: C(D_0) \to C(D_1)$ as in Proposition \ref{filtrations cobordism map}
respects the filtrations by $\Z \oplus \Z$ on the two complexes $C(D_0)$ and $C(D_1)$ up to a global shift. Therefore, such a chain map induces a graded morphism of spectral sequences $(c_r): (E_r(D_0)) \to (E_r (D_1))$, where $(E_r(D_i))$ is the spectral sequence converging to the Floer homology group $I^\natural(K_i)$. Each page after the Khovanov pages is a topological invariant, i.e. depends on the links $K_i$ only. In the proof of Proposition \ref{filtrations cobordism map} above in \cite{KM_filtrations} the chain map $c$ is in fact obtained by representing the cobordism $S$ by a movie $M$ between diagrams $D_0$ and $D_1$ for $K_0$ respectively $K_1$, and by then checking the claim for the map induced on reduced instanton knot Floer homology by the particular handle and Reidemeister moves.

One is tempted to believe that at the Khovanov page, the corresponding morphism $c(M)$ between the instanton Floer chain complexes as in the last Proposition is just equal to the map $\Khr(M): \Khr(K_0) \to \Khr(K_1)$ in Khovanov homology. In fact, such a functoriality property remains open in \cite{KM_filtrations}. What we {\em can} say, however, is that there is such a result in a particular situation.

\begin{proposition}\label{cobordism theorem}
	Let $D_0$ and $D_1$ be diagrams of knots $K_0$ and $K_1$. Let $S$ be a cobordism from $K_0$ to $K_1$ that is represented by a movie $M$ between the diagrams $D_0$ and $D_1$. Let us assume this movie consists only of isotopies of the diagrams (outside of balls containing the crossings) and handle attachment of index 1 (excluding Reidemeister moves). Then the map $I^\natural(S): I^\natural(K_0) \to I^\natural(K_1)$ is induced by a morphism of chain complexes
	\[
		c(M): (C(D_0),d_\natural(D_0)) \to (C(D_1), d_\natural(D_1))
	\]
respecting the bifiltration by $\Z \oplus \Z$, and this morphism induces the map $\Khr(M) : \Khr(K_0) \to \Khr(K_1)$ at the Khovanov page of the Kronheimer-Mrowka spectral sequence.
\end{proposition}
{\em Sketch of proof.}
This follows by inspecting the Proof of Proposition \ref{filtrations cobordism map} and of Theorem \ref{ss} and \ref{instanton cube} in the original source \cite{KM_filtrations}. As mentioned before, a general surface can be represented by a movie consisting of isotopies of the diagram, Reidemeister moves, and the addition of handles of index 0, 1 or 2. In our situation, we are just left with the morphism in instanton knot Floer homology induced by the 1-handle attachment. That this does the desired thing is just the consequence of Section 8 in \cite{KM_ss}, where it is shown that the Kronheimer-Mrowka spectral sequence starts at the Khovanov chain complex.
\qed

The corresponding result for the Ozsv\'ath-Szab\'o spectral sequence seems to be known and follows the same line of argument.

\begin{proposition}\label{cobordism theorem HF}
	  Let $D_0$, $D_1$, $K_0$, $K_1$, $S$, and $M$ be as above.  Working now of course with $\Z/2$ coefficients, the map $\HF(\Sigma(S)): \HF(K_0) \to \HF(K_1)$ is induced by a morphism of filtered chain complexes inducing the map $\Khr(M) : \Khr(K_0) \to \Khr(K_1)$ at the $E_2$-page of the Ozsv\'ath-Szab\'o spectral sequence. 
\end{proposition}
\qed

Suppose now that $S$ is a link cobordism between $K_0$ and $K_1$ with only index $1$ critical points.  In fact, it is not too hard to see that there exist diagrams $D_0$ and $D_1$ and a movie $M$ (presenting $K_0$, $K_1$, and $S$) which satisfy the requirements of the propositions above.  For those wanting details of how to construct such a movie $M$ we refer them to the proof of Theorem 1.6 of \cite{Lobb}.

\begin{equation*}
\begin{tikzcd}
 \Khr(K_0) \arrow{rr} & & \Khr(K_1) \arrow{ld} \\
	& \Khr(K_2) \arrow{ul} &
\end{tikzcd}
\end{equation*}

For knots $K_0$, $K_1$, $K_2$ related by the unoriented skein moves, the maps in the long exact sequence on Khovanov homology are each induced by some $1$-handle attachment up to Reidemeister-isomorphism.  Applying Proposition \ref{cobordism theorem} in this case gives us the following:

\begin{proposition}\label{exact triangle khovanov page}
	Suppose we are given knots or links $K_0, K_1$ and $K_2$ that only differ inside a ball by the unoriented skein moves, then there are obvious cobordisms $S_{01}$ from $K_0$ to $K_1$, $S_{12}$ from $K_1$ to $K_2$, and $S_{20}$ from $K_2$ to $K_0$ such that each $S_{ij}$ has a single critical point of index $1$. Then, fixing $i \in \{ 0, 1, 2 \}$, we can arrange that the map $I^\natural(S_{i,i+1}): I^\natural(K_i) \to I^\natural(K_{i+1})$ is induced by a filtered map on chain complexes 
	\begin{equation*}
		c : (C(D_i),d_\natural(D_i)) \to (C(D_{i+1}),d_\natural(D_{i+1})) , \  
	\end{equation*}
with $D_i$, $D_{i+1}$ being diagrams for $K_i$, $K_{i+1}$,
and such that the induced morphism between the resulting Kronheimer-Mrowka spectral sequences fits into an exact triangle at the Khovanov page relating $\Khr(K_0)$, $\Khr(K_1)$, and $\Khr(K_2)$.

\end{proposition}
\begin{proposition}
The analogue of the previous Proposition holds for the Ozsv\'ath-Szab\'o spectral sequence as well.
\end{proposition}

\hfill \qed

\section{Computations}

Essentially most of the arguments in this section progress by finding cobordisms between knots whose spectral sequences we wish to know and knots whose spectral sequences are necessarily trivial after the Khovanov page.

\subsection{The $(4,5)$ torus knot and instanton homology.}

In \cite{KM_filtrations} the example of the $(4,5)$ torus knot $T(4,5)$ is analysed and it is determined that there is exactly one non-trival differential in the spectral sequence after the Khovanov page.  Furthermore, it was shown that this differential would cancel exactly one of eight explicit pairs of generators in the Khovanov homology.

The technique in this paper almost determines the spectral sequence completely: we are able reduce the number of possible pairs to two, enabling us to give the filtration on $I^\natural(K)$ in almost all degrees.  More precisely, after proving Proposition \ref{45} we can write $I^\natural(K) = V^6 \oplus W^1$ where $V$ and $W$ are filtered vector spaces of dimensions $6$ and $1$ respectively and we know the filtration on $V$ completely and there are two possibilities for the filtration on $W$.

We note that none of the techniques currently available to constrain the filtered instanton homology (our technique included) discriminates between the various filtrations corresponding to choices $(a,b) \in \Z \oplus \Z$, $a,b \geq 1$.  \emph{A priori} it is possible that different choices of $(a,b)$ give different spectral sequences, and so it may be the case that, of the two possible canceling pairs in the spectral sequence for $T(4,5)$, each pair does in fact occur for different choices of filtration.

In the plot below we show the reduced Khovanov homology over $\Q$ of $T(4,5)$ as the solid discs.  The horizontal axis is the homological grading $i$, and we follow Kronheimer and Mrowka in making the vertical coordinate $j-i$ where $j$ is the quantum grading.  In \cite{KM_filtrations} it is shown that there is exactly one non-trivial differential in their spectral sequence.  

Using the $\Z/4$-grading on $I^\natural (K)$, Kronheimer-Mrowka showed that this differential will go from one of the three generators on line $j-i = 13$ to one of the three generators on the line $j-i = 16$.  The exception is that a differential from $(6, 13)$ to $(5,16)$ is impossible. In fact, as quoted in Theorem \ref{instanton cube}, the differential $d_\natural$ that computes the instanton Floer homology $I^\natural$ from a resolution cube in the spectral sequence preserves the descending quantum filtrations. Hence there are \emph{a priori} eight possible differentials.

\begin{proposition}
\label{45}
The differential in the Kronheimer-Mrowka spectral sequence for $T(4,5)$ either goes from the generator at bigrading $(2,13)$ to the generator at $(9,16)$ or goes from $(4,13)$ to $(9,16)$.
\end{proposition}

The proof of this proposition is broken into two lemmata.

\begin{lemma}
The generators at bigradings $(5,16)$ and $(7,16)$ are never boundaries after the Khovanov page in the spectral sequence.  Hence, since we know there is exactly one non-trivial differential, the generator at $(9,16)$ must be the boundary.
\end{lemma}

\begin{proof}
We consider the genus $1$ knot cobordism $\Sigma$ obtained as follows.  First express $T(4,5)$ as a braid closure.  Changing the sign of a crossing between the first two strands of the braid gives a knot that we shall call $K$.  Using time as the second coordinate, we have a cylinder embedded in $S^3 \times [0,1]$ with a single point of self-intersection which has boundary $T(4,5) \subset S^3 \times \{ 0 \}$ and $K \subset S^3 \times \{ 1 \}$.  Replacing the point of self-intersection with a piece of genus gives a knot cobordism $\Sigma$ between $T(4,5)$ and $K$.

We observe that the rank of reduced Khovanov homology of $K$ is $9$ and the sum of the absolute values of the Alexander polynomial of $K$ is also $9$.  Hence the Kronheimer-Mrowka spectral sequence associated to $K$ collapses at the Khovanov page by Proposition \ref{alex polynomial lower bound}.

\begin{center}
\begin{sseq}[grid=go, entrysize=6mm]{0...9}{9...16}
\ssmoveto{0}{9}
\ssdrop{\circle}
\ssmoveto{2}{11}
\ssdrop{\circle}
\ssmoveto{3}{12}
\ssdrop{\circle}
\ssmoveto{4}{11}
\ssdrop{\circle}
\ssmoveto{5}{14}
\ssdrop{\circle}
\ssmoveto{6}{13}
\ssdrop{\circle}
\ssmoveto{7}{14}
\ssdrop{\circle}
\ssmoveto{8}{15}
\ssdrop{\circle}
\ssmoveto{9}{16}
\ssdrop{\circle}
\ssmoveto{0}{11}
\ssdropbull
\ssarrow[curve=0.1]{0}{-2}
\ssmoveto{2}{13}
\ssdropbull
\ssarrow[curve=0.1]{0}{-2}
\ssmoveto{4}{13}
\ssdropbull
\ssarrow[curve=0.1]{0}{-2}
\ssmoveto{6}{13}
\ssdropbull
\ssmoveto{3}{14}
\ssdropbull
\ssarrow[curve=0.1]{0}{-2}
\ssmoveto{8}{15}
\ssdropbull
\ssmoveto{5}{16}
\ssdropbull
\ssarrow[curve=0.1]{0}{-2}
\ssmoveto{7}{16}
\ssdropbull
\ssarrow[curve=0.1]{0}{-2}
\ssmoveto{9}{16}
\ssdropbull
\end{sseq}
\end{center}

In this plot we show the reduced Khovanov homology over $\Q$ of $T(4,5)$ and of $K$.  The discs correspond to generators of the homology of $T(4,5)$, the circles to generators of the homology of $K$.

The cobordism $\Sigma$ being oriented, it induces a map on the Khovanov homology which preserves the homological grading and lowers the quantum grading by $2$.  Hence the rank of this map is at most $6$, and we have drawn the $6$ possibly non-zero components of this map.

We split the cobordism $\Sigma$ into the composition of two cobordisms $\Sigma_1$ and $\Sigma_2$ where $\Sigma_1$ is a cobordism obtained by adding a $1$-handle to $T(4,5)$ to obtain a 2-component link $L$, and where $\Sigma_2$ is obtained by adding a 1-handle to $L$ to obtain the knot $K$.  We can think of $L$ as being obtained by taking the vertical smoothing of a crossing between the first two strands of a standard braid presentation of $T(4,5)$.  When we replace the crossing in question by the horizontal smoothing we obtain the trefoil knot.

So we can use Proposition \ref{exact triangle khovanov page} to see that there exists a movie presentation of $\Sigma$ inducing a morphism of Kronheimer-Mrowka spectral sequences that at the Khovanov page is the composition of two maps

\[ \Khr(T(4,5)) \stackrel{\Khr(\Sigma_1)}{\rightarrow} \Khr(L) \stackrel{\Khr(\Sigma_2)}{\rightarrow} \Khr(K) \]

\noindent each of which has cone equal to the reduced Khovanov homology of the trefoil.  The rank of the reduced Khovanov homology of the trefoil knot is $3$, hence if the rank of $Kh(L)$ is $6+2b$ then the ranks of the maps $\Khr(\Sigma_1)$ and $\Khr(\Sigma_2)$ are both $6+b$ since

\[3 = {\rm rank}({\rm Cone} (\Khr(\Sigma_i))) = 9 + 6 + 2b - 2{\rm rank}(\Khr(\Sigma_i)) {\rm .}\]

\noindent Hence the rank of $\Khr(\Sigma)$ is at least $(6+b) + (6+b) - (6+2b) = 6$, but we have already seen that the rank is at most 6.

So we see by Lemma \ref{algebraic lemma} that, since the generators at $(5,16)$ and $(7,16)$ are mapped non-trivially under $\Khr(\Sigma)$, they are never boundaries after the Khovanov page, and hence they survive the spectral sequence.  Thus the generator at $(9,16)$ is the target of some non-zero differential.
\end{proof}

Next we try to narrow down the possible generators from which the differential of the spectral sequence emerges.

\begin{lemma}
The generator at bigrading $(6,13)$ on the Khovanov page is an $\infty$-cycle.
\end{lemma}

\begin{proof}
There is a cobordism topologically equivalent to a punctured Moebius band from the knot $5_2$ in Rolfsen's knot table to $T(4,5)$.  This is presented as a single 1-handle attachment in Figure \ref{5_2cob}.

\begin{figure}
\centerline{
{
\psfrag{thing1}{$= q^{1-n}$}
\psfrag{thing2}{$-q^{-n}$}
\psfrag{thing3}{$=q^{n-1}$}
\psfrag{thing4}{$-q^n$}
\psfrag{2}{$2$}
\psfrag{x4}{$x_4$}
\psfrag{T+(D)}{$T^+(D)$}
\includegraphics[height=4in,width=3in]{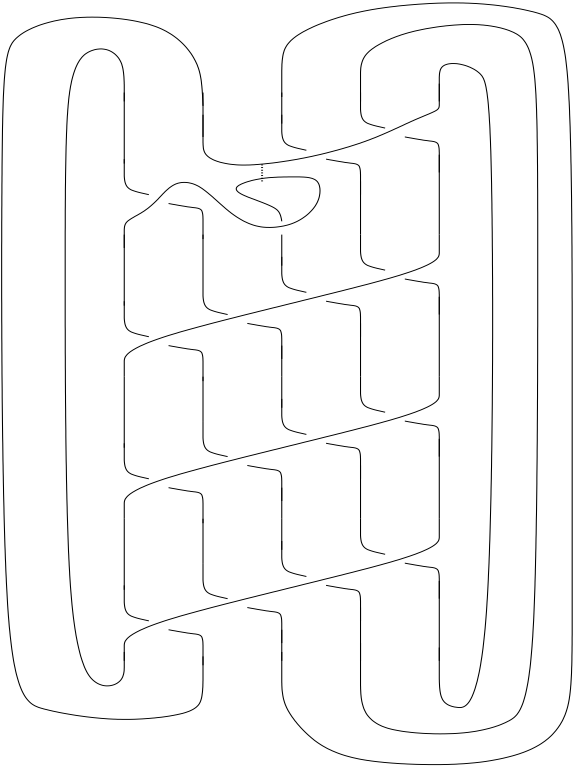}
}}
\caption{The dotted line represents a blackboard-framed 1-handle attachment which gives a knot cobordism from the knot $5_2$ to $T(4,5)$.}
\label{5_2cob}
\end{figure}

This 1-handle attachment induces a morphism of Kronheimer-Mrowka spectral sequences, which we are again able to compute explicitly on the Khovanov page.

On the Khovanov page the map raises the homological grading by $11$ and raises the quantum grading by $32$.  The possible non-zero components of this map are shown below.  In fact each component is non-zero, since the Khovanov homology of the cone (again, computed from the unoriented skein exact sequence) has homological width $2$.

\begin{center}
\begin{sseq}[grid=go, entrysize=4mm]{-5...9}{-8...16}
\ssmoveto{0}{11}
\ssdropbull
\ssmoveto{2}{13}
\ssdropbull
\ssmoveto{4}{13}
\ssdropbull
\ssmoveto{6}{13}
\ssdropbull
\ssmoveto{3}{14}
\ssdropbull
\ssmoveto{8}{15}
\ssdropbull
\ssmoveto{5}{16}
\ssdropbull
\ssmoveto{7}{16}
\ssdropbull
\ssmoveto{9}{16}
\ssdropbull
\ssmoveto{-5}{-8}
\ssdrop{\circle}
\ssarrow{11}{21}
\ssmoveto{-4}{-7}
\ssdrop{\circle}
\ssmoveto{-3}{-6}
\ssdrop{\circle}
\ssarrow{11}{21}
\ssmoveto{-2}{-5}
\ssdrop{\circle}
\ssmoveto{-2}{-5}
\ssdrop{\circle}
\ssarrow{11}{21}
\ssmoveto{-1}{-4}
\ssdrop{\circle}
\ssmoveto{0}{-3}
\ssdrop{\circle}
\end{sseq}
\end{center}

Since $5_2$ is alternating the spectral sequence has only trivial differentials past the Khovanov page.  This implies, again by Lemma \ref{algebraic lemma}, that the generator of the Khovanov homology of $T(4,5)$ that occurs at grading $(i,j-i) = (6,13)$ has to survive the spectral sequence (since it must be an $\infty$-cycle).
\end{proof}

Thus we have shown that there are only two remaining possibilities for the non-trivial differential in the spectral sequence, hence verifying Proposition \ref{45}.

\subsection{Three-stranded pretzel knots}
We will now apply our method to draw conclusions about the Floer homology of 3-stranded pretzel knots $P(p,q,r)$.  To avoid confusion we shall state this first about instanton homology and indicate at the end how the proof for Heegaard-Floer homology differs.

To avoid trivialities, we assume all of $p,q$ and $r$ are non-zero. Notice also that $P(p,q,r)$ is invariant under permutation of the numbers $p,q$ and $r$, and that reflection of $P(p,q,r)$ yields $P(-p,-q,-r)$. We will restrict ourselves to the cases where $P(p,q,r)$ is a knot, and this is so if and only if at most one of the numbers $p,q$ and $r$ is even.

 If the absolute value of one of the numbers $p,q$ and $r$ is $1$ then $P(p,q,r)$ is easily seen to be a 2-bridge link, hence alternating.  If all of $p,q,r$ are positive (or have the same sign), which we shall assume from now on, then the standard diagram of $P(p,q,r)$ is alternating.
 
By results of Greene \cite{Greene} and Champanerkar-Kofman \cite{CK} the 3-stranded pretzel links $P(-p,q,r)$ for $p,q,r \geq 2$ are {\em non quasi-alternating} if and only if $p \leq \min\{ q, r \}$.

Starkston has conjectured \cite{Starkston} and  Qazaqzeh \cite{Q} has shown that the Khovanov homology of $P(-p,q,r)$ is thin if $p = \min \{ q, r\}$, and Manion \cite{Manion} has proved that it is not thin if $ 2 \leq p < \min \{ q, r \}$. 

As a consequence of these results, the only 3-stranded pretzel knots which do not have their reduced instanton knot Floer homology determined by the  collapsing of the Kronheimer-Mrowka spectral sequence are the knots $P(-p,q,r)$ with $2 \leq p < \min \{q,r\}$. 

The following is a consequence of Manion's result \cite[Theorem 1.1]{Manion}.
\begin{prop}\label{Manion}
  Let $2 \leq p < \min\{q, r \}$. The reduced Khovanov homology of the pretzel knot $P(-p,q,r)$ is supported in two neighbouring $\delta$-gradings and is given by
  \begin{equation*}
  	\Khr(P(-p,q,r)) \cong \Q^{p^2-1} \oplus \Q^{(q-p)(r-p)-1} 
  \end{equation*}
  if all of $p,q,r$ are odd or only $q$ or $r$ is even, and is given by
    \begin{equation*}
  	\Khr(P(-p,q,r)) \cong \Q^{p^2} \oplus \Q^{(q-p)(r-p)} 
  \end{equation*}
  if $p$ is even. In both cases, the first summand denotes the reduced Khovanov homology of the upper $\delta$-grading and the second summand the one with the $\delta$-grading which is one lower. 
\end{prop}
Manion also makes precise the respective $\delta$-gradings, requiring further distinction of cases, but we do not need this here.  In fact, all we need is that any pretzel knot has reduced Khovanov homology supported in at most two adjacent $\delta$-gradings, which is easily proved by an induction.
\\

Before stating our theorem we prove a simple lemma. 
For any $n \geq 1$ the $T(2,2n)$ torus link is an alternating non-split two component link. Proposition \ref{knots with thin homology} does not immediately apply to conclude that the spectral sequence to $I^\natural(T(2,2n))$ is trivial even though it has thin homology because the excision isomorphism $I^\natural(K) \cong \KHI(K)$ is just stated for knots in \cite{KM_ss}.
\begin{lemma}
	The $T(2,2n)$ torus link has trivial Kronheimer-Mrowka spectral sequence from $\Khr(T(2,2n))$ to $I^\natural(T(2,2n))$. Both groups have total rank $2n$. 
\end{lemma}
\begin{proof}
	We use the exact triangle from Proposition \ref{triangle} twice. The torus knot $K=T(2,2n+1)$ in its standard diagram has a crossing such that the link $K_0$ resulting from the 0-resolution of that crossing is the torus link $T(2,2n)$, and the knot $K_1$ resulting from 1-resolution is the unknot $U$. For the ranks we have $\rk(\Khr(T(2,2n+1))) = 2n+1$ and $\rk(\Khr(U)) = 1$. Therefore, by the exact triangle the rank of $\Khr(T(2,2n))$ is either $2n+2$ or $2n$.
	
	The torus link $L=T(2,2n)$ has a crossing in its standard diagram such that the two resolutions are the torus knot $T(2,2n-1)$ and the unknot $U$, respectively. The exact triangle implies this time that the rank of $\Khr(T(2,2n))$ is either $2n$ or $2n-2$. Thus the rank of $\Khr(T(2,2n))$ is $2n$. 
	
	The torus knots $T(2,2n+1)$ are alternating, so have trivial spectral sequence. The exact triangle argument just above, but this time applied to reduced instanton homology, implies the claim.
\end{proof}

\begin{theorem}
	Let $2 \leq p < \min\{q, r \}$. Then for any filtration determined by numbers $a,b$ as above, the Kronheimer-Mrowka spectral sequence, starting from the reduced Khovanov homology $\Khr(P(-p,q,r))$ and abutting to the instanton knot Floer homology $I^\natural(P(-p,q,r))$, can only have non-trivial differentials that strictly lower the $\delta$-grading. 
	\end{theorem}
\begin{proof}

For the time being, we assume $q, r$ are odd numbers, and without loss of generality we can assume $q \leq r$. For any $p \geq 2$, 
the pretzel knots $P(-p,q,r)$, $P(-(p-1),q,r)$ and the torus link $T(2,q+r)$ are related by a skein triangle. By Proposition \ref{exact triangle khovanov page} there is a morphism of spectral sequences $\Psi$ from the Kronheimer-Mrowka spectral sequence for $P(-(p-1),q,r)$ to that of $P(-p,q,r)$ such that at the Khovanov page the morphism $\psi$ fits into an exact triangle:

\begin{equation} \label{triangle pretzels}
\begin{tikzcd}
 \Khr(P(-(p-1),q,r)) \arrow{rr}{\psi} & & \Khr(P(-p,q,r)) \arrow{ddl} \\ & & \\
	& \Khr(T(2,q+r))  . \arrow{uul}
\end{tikzcd}
\end{equation}

We will argue by induction in two different steps. In a first step, we show that all elements of the diagonal of the lower $\delta$-grading of $\Khr(P(-p,q,r))$
are $\infty$-cycles i.e. at any page no non-trivial differential of the spectral sequence emerges from this diagonal. In a second step we show that, at any page $E_k$ not before the Khovanov page, no element of the upper diagonal is a $k$-boundary.
\\

To start the induction of the first step, consider the case $p=2$. The knot $P(-1,q,r)$ is a 2-bridge knot and so its reduced Khovanov homology is thin. Likewise, the reduced Khovanov homology of the torus link $T(2,q+r)$ is thin. Therefore, the morphism $\psi$ maps {\em onto} one of the diagonals of $\Khr(P(-2,q,r))$. One can either check the grading shifts explicitly or argue by Manion's result that in fact $\psi$ maps onto the diagonal with the lower $\delta$-grading.

Now $P(-1,q,r)$ has trivial Kronheimer-Mrowka spectral sequence by Proposition \ref{knots with thin homology}, so there is no non-trivial differential, and, in particular, every element of $\Khr(P(-1,q,r))$ is an $\infty$-cycle. By Lemma \ref{algebraic lemma} and Proposition \ref{exact triangle khovanov page}, every element in the diagonal of the lower $\delta$-grading of $\Khr(P(-2,q,r))$ also is an $\infty$-cycle.

In the induction step for $3 \leq p \leq \min\{q,r\}$, the exact triangle (\ref{triangle pretzels}) will have two Khovanov homology groups with support in two diagonals, but the argument is similar:  The map $\psi$ has to map the diagonal with the lower $\delta$-grading of $\Khr(P(-(p-1),q,r))$ onto the diagonal with the lower $\delta$-grading of $\Khr(P(-p,q,r))$, for otherwise we obtain a contradiction to the fact that the Khovanov homology of the torus link $T(2,q+r)$ is thin. If by induction hypothesis every element of the lower diagonal of $\Khr(P(-(p-1),q,r))$ is an $\infty$-cycle, Lemma \ref{algebraic lemma} and Proposition \ref{exact triangle khovanov page} imply again that this also holds for the lower diagonal of $\Khr(P(-p,q,r))$.
\\

To start the induction of the second step, recall that $P(-q,q,r)$ has thin Khovanov homology by Qazaqzeh's result \cite{Q}. As a consequence, it has trivial Kron\-hei\-mer-Mrowka spectral sequence by Proposition \ref{knots with thin homology}. We apply Proposition \ref{exact triangle khovanov page} such that we obtain the exact triangle (\ref{triangle pretzels}) for $p=q$. The only possibility consistent with the fact that $P(-q,q,r)$ and $T(2,q+r)$ have thin Khovanov homology is that the morphism $\psi$ maps the diagonal with the {\em upper} $\delta$-grading of $P(-(q-1),q,r)$ injectively into the diagonal of $P(-q,q,r)$. By the same argument, but for $p < q$, we see that the only consistent possibility is that $\psi$ maps the diagonal with the upper $\delta$-grading of $P(-(p-1),q,r)$ injectively into the {\em upper} diagonal of $P(-p,q,r)$. In both cases, what we need is that $\psi$ will be an injective map.

By composition of these various maps $\psi$ we get a map

\[ \phi : \Khr(P(-p,q,r)) \rightarrow \Khr(P(-q,q,r)) \]

\noindent which is injective when restricted to the upper diagonal of $\Khr(P(-p,q,r))$ and by Proposition \ref{exact triangle khovanov page} is the induced map at the Khovanov page of a morphism between the Kronheimer-Mrowka spectral sequences for the knots $P(-p,q,r)$ and $P(-q,q,r)$.

Suppose the $a^{\rm th}$ page is the Khovanov page, and let $x \in \Khr(P(-p,q,r)) = E_a$ be an $a$-cycle such that each $[x]_s$ is an $s+1$-cycle for $a \leq s \leq b-1$ and suppose $x$ lies in the upper $\delta$-grading diagonal of $\Khr(P(-p,q,r))$.  Then if $[x]_b$ is a $b$-boundary we have that

\[ \phi_b [x]_b = [\phi (x)]_b \]

\noindent is also a $b$-boundary by Lemma \ref{algebraic lemma}.  Hence $\phi(x) = 0$ because the spectral sequence for $P(-q,q,r)$ collapses at the Khovanov page, hence $x = 0$ since $\phi$ is injective when restricted to the upper diagonal.

So far we have proved the theorem for all cases where both $q$ and $r$ are odd numbers. Assume now without loss of generality that $q$ is even and $p$ and $r$ are both odd.

The pretzel knots $P(-p,q,r)$, $P(-p,q+1,r)$ and the torus link $T(2,r-p)$ also form an exact triangle to which we apply Proposition \ref{exact triangle khovanov page} and Lemma \ref{algebraic lemma} another time.  There is a morphism of spectral sequences $\Psi$ from the Kronheimer-Mrowka spectral sequence of $P(-p,q+1,r)$ to that of $P(-p,q,r)$ such that the morphism $\psi$ at the Khovanov page fits into the exact triangle
\begin{equation*}
\begin{tikzcd}
 \Khr(P(-p,q+1,r)) \arrow{rr}{\psi} & & \Khr(P(-p,q,r)) \arrow{ddl} \\ & & \\
	& \Khr(T(2,r-p))  . \arrow{uul}
\end{tikzcd}
\end{equation*}
Again, this morphism has to map the diagonal with the lower $\delta$-grading of the group $\Khr(-p, q+1, r)$ onto the lower diagonal of $\Khr(P(-p,q,r))$, and 
we can use the theorem for the pretzel knot $P(-p,q+1,r)$ to draw the conclusion that there is no non-trivial differential when restricted to the diagonal with the lower $\delta$-grading, at any page of the Kronheimer-Mrowka spectral sequence for $P(-p,q,r)$.

Similarly, the pretzel knots $P(-p,q,r$, $P(-p,q-1,r)$ and the torus link $T(2,r-p)$ also form an exact triangle to which we apply Proposition \ref{exact triangle khovanov page}.  There is a morphism of spectral sequences $\Psi$ from the Kronheimer-Mrowka spectral sequence of $P(-p,q,r)$ to that of $P(-p,q-1,r)$ such that the morphism $\psi$ at the Khovanov page fits into the exact triangle
\begin{equation*}
\begin{tikzcd}
 \Khr(P(-p,q,r)) \arrow{rr}{\psi} & & \Khr(P(-p,q-1,r)) \arrow{ddl} \\ & & \\
	& \Khr(T(2,r-p))  . \arrow{uul}
\end{tikzcd}
\end{equation*}
Again, this morphism has to map the diagonal with the upper $\delta$-grading of the group $\Khr(-p,q,r)$ injectively into the upper one of $\Khr(P(-p, q-1, r))$. Using the same method as before, we can use the theorem for the pretzel knot $P(-p,q-1,r)$ to draw the conclusion that no element of the diagonal with the upper $\delta$-grading lies in the image of a differential, at any page of the Kronheimer-Mrowka spectral sequence for $P(-p,q,r)$. 
\end{proof}

\begin{theorem}
Let $2 \leq p < \min\{q, r \}$. Then the Ozsv\'ath-Szab\'o spectral sequence, starting from the reduced Khovanov homology $\Khr(P(-p,q,r))$ and abutting to the Heegaard-Floer homology of the branched double cover of the mirror $\HF(\Sigma(K))$, can only have non-trivial differentials that strictly lower the $\delta$-grading. 
\end{theorem}

\begin{proof}
The only substantial difference is that we are now working over the 2-element field $\Z/2$.  With these coefficients one may be worried that the reduced Khovanov homology of a pretzel link may not be supported in two adjacent $\delta$-gradings, but this turns out not to be the case, for example by appealing to Manion's result \cite{Manion} in which he proved that the reduced Khovanov homology of a pretzel knot over $\Z$ is torsion-free.
\end{proof}

\begin{remark}
We have observed above that for the torus knot $T(4,5)$ the same conclusion holds: all possible non-zero differentials in the Kronheimer-Mrowka spectral sequence strictly lower the $\delta$-grading.
\end{remark}

Based on these results we state the following conjecture.

\begin{conjecture}
For any knot, all non-trivial differentials in the Kronheimer-Mrowka spectral sequence strictly lower the $\delta$-grading.
\end{conjecture}

\begin{prop}
	The suite of pretzel knots $P(-2,3,2n+1)$ all have trivial Kron\-hei\-mer-Mrowka spectral sequence.
\end{prop}
\begin{proof}
The sum of the absolute values of the coefficients of the Alexander polynomial of $P(-2,3,2n+1)$ is equal to $2n+3$. Therefore, $I^\natural(P(-2,3,2n+1))$ has rank bounded below by $2n+3$ by Proposition \ref{alex polynomial lower bound}. On the other hand, Manion's result says that the rank of $\Khr(P(-2,3,2n+1))$ is also equal to $2n+3$. Hence, the Kronheimer-Mrowka spectral sequence is trivial.
\end{proof}

It is not the case that these knots $P(-2,3,2n+1)$ have trivial Ozsv\'ath-Szab\'o spectral sequence.  In fact, in Proposition \ref{Baldwin_turbo_charged} we determine explicitly the Ozsv\'ath-Szab\'o spectral sequences for these knots.

\subsection{The $(-2,3,2n+1)$ pretzel knots and the Ozsv\'ath-Szab\'o spectral sequence.}

We have seen earlier that the Kronheimer-Mrowka spectral sequence collapses at the Khovanov page for all pretzel knots $P(-2,3,2n+1)$.  This however is not the case for the Ozsv\'ath-Szab\'o spectral sequence.

In this subsection we work over $\Z/2$.  In \cite{Baldwin}, Baldwin considered the pretzel knot $P(-2,3,5)$ and determined the pages of the Ozsv\'ath-Szab\'o spectral sequence from the reduced Khovanov homology of $P(-2,3,5)$ to the Heegaard-Floer homology of the branched double cover (which in this case is the Poincar\'e homology 3-sphere).

The reduced Khovanov homology of $P(-2,3,5)$ is given by

\[ \Khr(P(-2,3,5)) = t^0q^8 + t^2q^{12} + t^3q^{14} + t^4q^{14} + t^5q^{18} + t^6q^{18} + t^7q^{20} {\rm , }\]

\noindent where we have been cavalier about the distinction between the homology groups and the Poincare polynomial (we shall continue to be cavalier). Baldwin showed that $h^0q^8$ survives the spectral sequence and the remaining six elements cancel in pairs:

\[ (t^2q^{12}, t^4q^{14}), (t^5 q^{18}, t^7 q^{20}), (t^3 q^{14}, t^6 q^{18}) {\rm ,} \]

\noindent where the first two pairs cancel from the $E_2$ page to the $E_3$ page, and the third pair cancel from the $E_3$ to the $E_4$ page.

Manion's result \cite{Manion} implies that the reduced Khovanov homology of $P(-2,3,2n+1)$ is supported in two adjacent delta gradings (where $\delta$ is defined as half the quantum grading minus the homological grading).  It has rank $2n-1$ in delta grading $\delta = n+1$ and rank $4$ in delta grading $\delta = n+2$. In fact, we can write

\begin{equation*}
\begin{split}
 \Khr(P(-2,3,2n+1)) = & q^{2n-4} \Khr(P(-2,3,5)) \\ & +  t^8q^{2n + 18}(1 + tq^2 + (tq^2)^2 + \cdots + (tq^2)^{2n-5}) {\rm ,} 
\end{split}
\end{equation*}
\noindent where $n \geq 3$. We write this as

\[ \Khr(P(-2,3,2n+1)) = q^{2n-4} \Khr(P(-2,3,5)) \oplus T_n {\rm ,} \]

\noindent where $T_n$ stands for \emph{Tail}.  We note that for degree reasons each bihomogenous element of $\Khr(P(-2,3,2n+1))$ lies either in $q^{2n-4} \Khr(P(-2,3,5))$ or in $T_n$.

\begin{proposition}
\label{Baldwin_turbo_charged}
The Ozsv\'ath-Szab\'o spectral sequence for the knot $P(-2,3,2n+1)$ is obtained by shifting the spectral sequence for $P(-2,3,5)$ by $q^{2n-4}$ and taking the direct sum with a trivial spectral sequence given by $T_n = E_2 = E_\infty$.
\end{proposition}

\begin{proof}
Firstly we want to see that the element $t^0q^{2n+4}$ in $\Khr(P(-2,3,2n+1))$ has to survive the spectral sequence.  To see this we just observe that Baldwin's argument for the case $n=2$ actually works for $n \geq 2$.  Essentially since $P(-2,3,2n+1)$ is a positive knot and the Khovanov homology is of rank $1$ in homological degree $0$, it follows that Plamenevskaya's element \cite{Plamenevskaya} is exactly the element $t^0q^{2n+4}$.  Baldwin shows that Plamenevskaya's element represents a cycle in every page of the spectral sequence and, since the Ozsv\'ath-Szab\'o differentials always increase the homological grading, this implies that $t^0q^{2n+4}$ survives to the $E_\infty$ page.

Next we note that there is a orientable knot cobordism induced by the addition of $2n-4$ $1$-handles from $P(-2,3,2n+1)$ to $P(-2,3,5)$.  Now this induces a map on Khovanov homologies

\[ \phi: \Khr(P(-2,3,2n+1)) \rightarrow \Khr(P(-2,3,5)) {\rm ,} \]

\noindent such that

\begin{itemize}
\item $\phi$ is of bidegree $(0,4-2n)$,
\item $\phi$ is the map on the $E^2$-pages of a morphism between the Ozsv\'ath-Szab\'o spectral sequences of the two knots,
\item $\phi$ is onto.
\end{itemize}

\noindent The last bullet point follows from the unoriented skein exact sequence in Khovanov homology and a comparison of ranks.

As an example, below we have drawn the Khovanov homology of $P(-2,3,11)$.  The discs correspond to generators whose image under $\phi$ is non-zero, the circles are generators in the kernel of $\phi$.  The arrows are the higher differentials of the spectral sequence which we are trying to prove exist.

\begin{center}
\begin{sseq}[grid=go, entrysize=3mm]{0...14}{14...38}
\ssmoveto{0}{14}
\ssdropbull
\ssmoveto{2}{18}
\ssdropbull
\ssmoveto{3}{20}
\ssdropbull
\ssmoveto{4}{20}
\ssdropbull
\ssmoveto{5}{24}
\ssdropbull
\ssmoveto{6}{24}
\ssdropbull
\ssmoveto{7}{26}
\ssdropbull
\ssmoveto{8}{28}
\ssdrop{\circle}
\ssmoveto{9}{30}
\ssdrop{\circle}
\ssmoveto{10}{32}
\ssdrop{\circle}
\ssmoveto{11}{34}
\ssdrop{\circle}
\ssmoveto{12}{36}
\ssdrop{\circle}
\ssmoveto{13}{38}
\ssdrop{\circle}
\ssmoveto{2}{18}
\ssarrow[curve=0]{2}{2}
\ssmoveto{5}{24}
\ssarrow[curve=0]{2}{2}
\ssmoveto{3}{20}
\ssarrow[curve=0]{3}{4}
\end{sseq}
\end{center}

Now each differential on the $n$th page $E_n$ in the Ozsv\'ath-Szab\'o spectral sequence raises the homological grading by $n$.  We know from the previous section that each differential in the spectral sequence for a pretzel knot has to lower the delta grading by $1$.  Hence each differential on the $n$th page is of bidegree $(n,2(n-1))$.

Let us now look at the map between the $E_2$ pages, we have the following commutative diagram:

\begin{equation*} \begin{CD}
t^2 q^{2n+8} @>d_2>> t^4 q^{2n + 10} \\
@VV\phi V   @VV\phi V \\
t^2 q^{12} @>D_2>> t^4 q^{14}
\end{CD} \end{equation*}

\noindent where the bottom row is part of the $E_2$ page for $P(-2,3,5)$ and the top row is part of the $E_2$ page for $P(-2,3,2n+1)$.  The differential $d_2$ is forced to be non-zero since all other arrows are non-zero.  Hence $(t^2 q^{2n+8} , t^4 q^{2n + 10})$ is a canceling pair on the $E_2$ page for $P(-2,3,2n+1)$.  A similar argument tells us that $(t^5 q^{2n + 14},t^7 q^{2n + 16})$ is another canceling pair on the $E_2$ page.

Now we look at the $E_3$ page.  Again the bottom row is $P(-2,3,5)$, the top row is $P(-2,3,2n+1)$.

\[ \begin{CD}
[t^3 q^{2n+10}]_3 @>d_3>> [t^6 q^{2n + 14}]_3 \\
@VV[\phi]_3 V   @VV[\phi]_3 V \\
t^3 q^{14} @>D_3>> t^6 q^{18}
\end{CD}\]

The bottom row is just the differential that we know exists on the $E_3$ page for $P(-2,3,5)$.  The arrows labelled $[\phi]_3$ are components of the map induced by the map $\phi$ between the two $E_2$ pages.  The terms labelled $[t^3 q^{2n+10}]_3$ and $[t^6 q^{2n + 14}]_3$ are the images in the $E_3$ page of two generators of the $E_2$ page and $d_3$ is a potentially non-zero differential between them.  In fact it is clear from the commutativity of the diagram that $d_3 \not= 0$ so long as both $[t^3 q^{2n+10}]_3 \not= 0$ and $[t^6 q^{2n + 14}]_3 \not= 0$.  And this is certainly true since there are no generators of the $E_2$ page of the spectral sequence for $P(-2,3,2n+1)$ with the correct bidigrees to cancel with these generators at that page.  Hence $d_3 \not= 0$ and $(t^3 q^{2n+10}, t^6 q^{2n + 14})$ is a canceling pair at the $E_3$ page.

We note that there is no homogenous generator in the tail $T_n$ with the correct bidegree to cancel before the $E_4$ page.

It remains to see that this is where the spectral sequence for $P(-2,3,2n+1)$ ends: $E_4 = E_\infty$.  We are left at the $E_4$ page with

\[ E_4 = t^0q^{2n+4} \oplus T_n {\rm .} \]

\noindent There can be no canceling pair entirely within $T_n$ since $T_n$ is supported in a single delta grading.  Furthermore we already know that $t^0q^{2n+4}$ survives the spectral sequence.
\end{proof}

\subsection{The $(-3,5,7)$ pretzel knot}

The pretzel knot $P(-3,5,7)$ has trivial Alexander polynomial $\Delta(P(-3,5,7)) = 1$.  The rank of the reduced Khovanov homology $\Khr(P(-3,5,7))$ is $15$, hence \emph{a priori} the rank of $I^\natural(P(-3,5,7))$ is some odd integer between $1$ and $15$.  Since $I^\natural$ detects the unknot we can immediately do a little better and exclude the possibility that $I^\natural(P(-3,5,7))$ has rank $1$!

It is not too hard in fact to see that the rank of $I^\natural(P(-3,5,7))$ is at least $11$, simply by using the long exact sequence

\begin{equation*}
\begin{tikzcd}
 I^\natural(P(-3,6,7)) \arrow{rr} & & I^\natural(P(-3,5,7)) \arrow{ddl} \\ & & \\
	& I^\natural(T(2,4))  . \arrow{uul}
\end{tikzcd}
\end{equation*}

\noindent (where we write $T(2,4)$ for the $(2,4)$ torus link), and computing that the the rank of $I^\natural(P(-3,6,7))$ has to be at least $15$ since that is the sum of the absolute values of the coefficients of its Alexander polynomial.

In this subsection we consider the problem of attempting to restrict the possible differentials of the Kronheimer-Mrowka spectral sequence of $P(-3,5,7)$ in order to deduce more about the filtrations on $I^\natural(P(-3,5,7))$.  This is to illustrate that the techniques of this paper can give more information on the Kronheimer-Mrowka spectral sequence of a pretzel knot than just that they decrease the $\delta$-grading.

First we consider the unoriented skein long exact triangle in Khovanov homology induced by taking resolutions of a crossing in the second of the three twisted regions.

This induces a long exact triangle of the following form:

\begin{equation*}
\begin{tikzcd}
 \Khr(P(-3,5,7)) \arrow{rr} & & \Khr(P(-3,4,7)) \arrow{ddl} \\ & & \\
	& \Khr(T(2,4)) \arrow{uul}
\end{tikzcd}
\end{equation*}

\noindent We compute the ranks $\vert \Khr(P(-3,4,7)) \vert = 11$ and $\vert \Khr(T(2,4)) \vert = 4$ and the sums of absolute values of coefficients of the Alexander polynomial $\vert \Delta(P(-3,4,7)) \vert = 11$.  Hence we conclude that the spectral sequence for $P(-3,4,7)$ is trivial and moreover that the map $\Khr(P(-3,5,7)) \rightarrow \Khr(P(-3,4,7))$ is of rank $11$.  By considering the bidegree of this map we can write down the bigradings of $11$ linearly independent bigraded generators of $\Khr(P(-3,5,7))$ which are mapped to non-zero elements of $\Khr(P(-3,4,7))$.

Below we have drawn the bigrading of $\Khr(P(-3,5,7))$ ($i$ along the horizontal axis and $j-i$ in the vertical direction) and indicated by solid discs these $11$ generators.  Since we know by Proposition \ref{exact triangle khovanov page} that this map $\Khr(P(-3,5,7)) \rightarrow \Khr(P(-3,4,7))$ can be realized as the induced map at the Khovanov page of a morphism between the two Kronheimer-Mrowka spectral sequences, we can apply Lemma \ref{algebraic lemma}.  Since the spectral sequence for $P(-3,4,7)$ is trivial, none of these elements represented by solid discs can be the target of differentials in the spectral sequence for $P(-3,5,7)$.  We have drawn circles to indicate the $4$ remaining bigraded generators of $\Khr(P(-3,5,7))$ which may be targets of differentials in the spectral sequence.

\begin{center}
\begin{sseq}[grid=go, entrysize=4mm]{-3...9}{5...15}
\ssmoveto{-3}{5}
\ssdropbull
\ssmoveto{-2}{6}
\ssdropbull
\ssmoveto{-1}{7}
\ssdropbull
\ssmoveto{0}{8}
\ssdropbull
\ssdropbull
\ssmoveto{1}{9}
\ssdropbull
\ssmoveto{2}{10}
\ssdropbull
\ssmoveto{3}{11}
\ssdropbull
\ssmoveto{4}{10}
\ssdropbull
\ssmoveto{5}{11}
\ssdrop{\circle}
\ssmoveto{6}{12}
\ssdropbull
\ssmoveto{7}{13}
\ssdropbull
\ssdrop{\circle}
\ssmoveto{8}{14}
\ssdrop{\circle}
\ssmoveto{9}{15}
\ssdrop{\circle}
\end{sseq}
\end{center}

Next we consider the long exact sequence in Khovanov homology obtained by resolving a crossing in the first of the three twisted regions of $P(-3,5,7)$.  This gives a long exact sequence of the form:

\begin{equation*}
\begin{tikzcd}
 \Khr(P(-3,5,7)) \arrow{rr} & & \Khr(T(2,12)) \arrow{ddl} \\ & & \\
	& \Khr(P(-2,5,7)) \arrow{uul}
\end{tikzcd}
\end{equation*}

We compute ranks $\vert \Khr(P(-2,5,7)) \vert = 19$ and $\vert \Khr(T(2,12)) \vert = 12$ and the sum of absolute values of the coefficients $\vert \Delta (P(-2,5,7))\vert = 19$.  Hence we can conclude that the spectral sequence for $P(-2,5,7)$ is trivial and that the rank of the map $\Khr(P(-2,5,7)) \rightarrow \Khr(P(-3,5,7))$ is $11$.

By considering the bigraded degree of this map we can give the bigradings of a bigraded basis for its image, none of whose elements can be sources of non-trivial differentials in the Kronheimer-Mrowka spectral sequence for $P(-3,5,7)$ (again by Proposition \ref{exact triangle khovanov page} and Lemma \ref{algebraic lemma}).  In the diagram below we have indicated by circles the bigradings of the remaining $4$ bigraded generators of $\Khr(P(-3,5,7))$ which may be the source of non-trivial differentials in the spectral sequence.

\begin{center}
\begin{sseq}[grid=go, entrysize=4mm]{-3...9}{5...15}
\ssmoveto{-3}{5}
\ssdrop{\circle}
\ssmoveto{-2}{6}
\ssdropbull
\ssmoveto{-1}{7}
\ssdrop{\circle}
\ssmoveto{0}{8}
\ssdrop{\circle}
\ssdropbull
\ssmoveto{1}{9}
\ssdropbull
\ssmoveto{2}{10}
\ssdrop{\circle}
\ssmoveto{3}{11}
\ssdropbull
\ssmoveto{4}{10}
\ssdropbull
\ssmoveto{5}{11}
\ssdropbull
\ssmoveto{6}{12}
\ssdropbull
\ssmoveto{7}{13}
\ssdropbull
\ssdropbull
\ssmoveto{8}{14}
\ssdropbull
\ssmoveto{9}{15}
\ssdropbull
\end{sseq}
\end{center}

We now consider the $\Z/4$-grading which is just the reduction modulo $4$ of the grading $j-i$.  Any non-trivial differential in the spectral sequence changes this grading by $3$ modulo $4$.  We can conclude that there are at most $4$ possibilities for differentials in the spectral sequence, at most two of which can actually occur.  We have drawn these four possibilities in the diagram below.

\begin{center}
\begin{sseq}[grid=go, entrysize=5mm]{-3...9}{5...15}
\ssmoveto{-3}{5}
\ssdropbull
\ssmoveto{-2}{6}
\ssdropbull
\ssmoveto{-1}{7}
\ssdropbull \ssname{f}
\ssmoveto{0}{8}
\ssdropbull \ssname{a}
\ssdropbull
\ssmoveto{1}{9}
\ssdropbull
\ssmoveto{2}{10}
\ssdropbull \ssname{c}
\ssmoveto{3}{11}
\ssdropbull
\ssmoveto{4}{10}
\ssdropbull
\ssmoveto{5}{11}
\ssdropbull \ssname{e}
\ssmoveto{6}{12}
\ssdropbull
\ssmoveto{7}{13}
\ssdropbull \ssname{b}
\ssdropbull
\ssmoveto{8}{14}
\ssdropbull \ssname{g}
\ssmoveto{9}{15}
\ssdropbull \ssname{d}
\ssmoveto{-1}{7}
\ssgoto f \ssgoto g \ssstroke[curve=0]
\ssgoto a \ssgoto e \ssstroke[curve=0]
\ssgoto a \ssgoto d \ssstroke[curve=0.3]
\ssgoto c \ssgoto b \ssstroke[curve=0]
\end{sseq}
\end{center}

We observe that there are $8$ generators which certainly survive the spectral sequence and whose bigradings we know explicitly.

\subsection{Relation to representation spaces}
Given a knot $K$ and a meridian $m$ of $K$, one may define the space of representations
\[
	R(K;{\bf i}) = \{ \rho  \in \, \Hom(\pi_1(S^3 \setminus K), SU(2)) \, | \,\tr(\rho(m)) = 0 \} \ .
\]
Here we also denote by $m$ the class of a meridian in $\pi_1(S^3 \setminus K)$, well defined up to conjugacy, and a representation is required to send this element to the conjugacy class of traceless matrices in $SU(2)$.

Reduced instanton knot Floer homology $I^\natural(K)$ of a knot $K$ is by definition the homology of a complex $(C(K)^\natural,d^\natural)$. This is, in some sense, the Morse homology of a Chern-Simons functional, suitably perturbed so as to obtain transversality of the involved instanton moduli spaces. The critical space of the {\em unperturbed} functional is related to the space $R(K;{\bf i})$ as follows (see \cite{KM_ss,KM_sutures,HHK}): Each conjugacy class of an irreducible representation in $R(K;{\bf i })$ accounts for a circle, and the conjugacy class of the reducible representation accounts for a point. 

In the most generic situation, $R(K;{\bf i})$ consists of only finitely many conjugacy classes. In this situation, after perturbation of the Chern-Simons functional, each critical circle is expected to yield two critical points. This has been described explicitly by Hedden, Herald, and Kirk in \cite{HHK} in a quite general setting. In this situation, the complex $C(K)^\natural$ is a free $\Q$-vector space of dimension $1+2n$, where $n$ is the number of conjugacy classes of irreducible representations in $R(K;{\bf i})$. The reduced instanton homology $I^\natural(K)$ is then bounded above by $1+2n$ as well.
\\

It is an interesting fact that the upper bound from Khovanov homology seems to be better than the upper bound from the representation space for pretzel knots, whereas for torus knots the converse seems to be the case in general (except for the torus knots $T(3,n)$). We list a few cases explicitly. The claims on the representation spaces of pretzel knots can be found in \cite{FKPC} and \cite{Z}. \\

{\em 
	\begin{itemize}
	\item For the pretzel knot $P(-3,5,7)$ we have $\rk(\Khr(P(-3,5,7))) = 15$, whereas $R(P(-3,5,7);{\bf i})$ contains the conjugacy class of the reducible and 16 conjugacy classes of irreducible non binary dihedral representations (see the table of the example in \cite{Z} where 3 errors occur that yield a total error of 1 which multiplied by two gave the wrong claim of 18 conjugacy classes). \\
	\item For the pretzel knots $P(-2,3,2n+1)$ we have $\rk(\Khr(P(-2,3,2n+1))) = 2n+3$. The representation space $R(P(-2,3,2n+1);{\bf i})$ contains the conjugacy class of the reducible representation, $2n-6$ irreducible binary dihedral representations, and $\lfloor \frac{8}{3} n \rfloor$ conjugacy classes of irreducible non binary dihedral representations, therefore yielding an upper bound to $I^\natural(P(-2,3,2n+1))$ by $\lfloor (4 + \frac{2}{3}) n -5 \rfloor$.\\ 
	\item Torus knots $T(p,q)$ with $p,q \geq 4$ seem to have a faster growth in reduced Khovanov homology than in the bound coming from representation spaces, see \cite[Section 12.5]{HHK}.
	\end{itemize}	
}


\begin{thebibliography}{99999}

\bibitem{Baldwin} J.~Baldwin, On the spectral sequence from Khovanov homology to Heegaard Floer homology, \emph{Int. Math. Res. Not.} 15 (2011), 3426--3470.

\bibitem{CK} A. Champanerkar, I. Kofman,
Twisting quasi-alternating links. \emph{ Proc. Amer. Math. Soc.} 137 (2009), no. 7, 2451--2458. 

\bibitem{FKPC} Y.~Fukumoto, P.~Kirk, J.~Pinz\'on-Caicedo, Traceless SU(2) representations of 2-stranded tangles, arXiv:1305.6042

\bibitem{Greene} J. Greene, Homologically thin, non-quasi-alternating links. \emph{Math. Res. Lett.} 17 (2010), no. 1, 39--49. 

\bibitem{HHK} M. Hedden, C. Herald, P. Kirk, The pillowcase and perturbations of traceless representations of knot groups, arXiv:1301.0164 


\bibitem{Khovanov_patterns} M.~Khovanov, 
Patterns in knot cohomology. I. \emph{
Experiment. Math.} 12 (2003), no. 3, 365--374. 

\bibitem{KM_sutures} P.~B.~Kronheimer and T.~S.~Mrowka, Knots, sutures, and excision, \emph{J. Differential Geom.} 84 (2010), no. 2, 301--364.

\bibitem{KM_alex} P.~B.~Kronheimer and T.~S.~Mrowka, Instanton Floer homology and the Alexander polynomial. \emph{Algebr. Geom. Topol.} 10 (2010), no. 3, 1715--1738.

\bibitem{KM_ss} P.~B.~Kronheimer and T.~S.~Mrowka, Khovanov homology is an unknot-detector, \emph{Publ. Math. Inst. Hautes Etudes Sci.} No. 113 (2011), 97--208.

\bibitem{KM_filtrations} P.~B.~Kronheimer and T.~S.~Mrowka, Filtrations on instanton knot homology, \emph{Preprint 2011}

\bibitem{Lim} Y.~Lim, Instanton homology and the Alexander polynomial,\emph{Proc. Amer. Math. Soc.} 138 (2010), no. 10, 3759--3768. 


\bibitem{Lobb} A. ~Lobb, A slice genus lower bound from sl(n) Khovanov-Rozansky homology,  \emph{Adv. Math.}  222  (2009), 1220--1276.

\bibitem{Manion} A. Manion, The Khovanov homology of 3-strand pretzels, revisited, arXiv:1303.3303

\bibitem{OS_first} P.~Ozsv\'ath and Z.~Szab\'o,
Holomorphic disks and topological invariants for closed three-manifolds, \emph{
Ann. of Math.} (2) 159 (2004), no. 3, 1027--1158. 

\bibitem{OSb2c} P.~Ozsv\'ath and Z.~Szab\'o, On the Heegaard-Floer homology of branched double-covers, \emph{Adv. Math.}, No. 94 (2005), 1--33.

\bibitem{Plamenevskaya} O.~Plamenevskaya, Transverse knots and Khovanov homology, \emph{Math. Res. Lett.}, 13(4) (2006) 571--586.

\bibitem{Q}  K.~Qazaqzeh, The Khovanov homology of a family of three-co
lumn pretzel links. \emph{
Commun.
Contemp. Math.}
, 13(5):813--825, 2011.

\bibitem{Starkston} L. Starkston, The Khovanov homology of $(p,-p,q)$ pretzel knots. \emph{
J. Knot Theory Ramifications}, 21 (2012), no. 5, 1250056, 14 pp. 

\bibitem{Z} R.~Zentner, Representation spaces of pretzel knots, \emph{Algebraic \& Geometric Topology} 11 (2011) 2941--2970.

\end{thebibliography}
\end{document}